\documentclass[12pt]{article}
\title{Finite totally $k$-closed groups}
\usepackage{authblk}
\author[1,2]{Dmitry Churikov}
\author[3]{Cheryl E Praeger}
\affil[1]{Sobolev Institule of Mathematics, Novosibirsk, Russia}
\affil[2]{Novosibirsk State University, Novosibirsk, Russia}
\affil[3]{Centre for the Matheamtics of Symmetry and Computation\\
University of Western Australia, Perth, WA, Australia}

\usepackage{amsmath,amssymb,latexsym,amsthm,enumerate,color}

\newtheorem{theorem}{Theorem}[section]
\newtheorem{lemma}[theorem]{Lemma}

\newtheorem{remark}[theorem]{Remark}
\newtheorem{problem}{Problem}

\theoremstyle{definition}

\def\Om{\Omega}

\def\mZ{{\mathbb Z}}

\def\Sym{{\rm Sym}}
\def\Orb{{\rm Orb}}

\def\Syl{{\rm Syl}}

\begin{document}
\maketitle

\begin{abstract}
For a positive integer $k$, a group $G$ is said to be totally $k$-closed if in each of its faithful permutation representations, say on a set $\Om$, $G$ is the largest subgroup of $\Sym(\Om)$ which leaves invariant each of the $G$-orbits in the induced action on $\Om\times\dots\times \Om=\Om^k$. We prove that every abelian group $G$ is totally $(n(G)+1)$-closed, but is not totally $n(G)$-closed, where $n(G)$ is the number of invariant factors in the invariant factor decomposition of~$G$. In particular, we prove that for each $k\geq2$ and each prime $p$, there are infinitely many finite abelian $p$-groups which are totally $k$-closed but not totally $(k-1)$-closed.
This result in the special case $k=2$ is due to Abdollahi and Arezoomand. We pose several open questions about total $k$-closure.
\end{abstract}

\section{Introduction}\label{intro}

In 1969 Wielandt~\cite[Definition 5.3]{Wielandt1969} introduced, for each positive integer $k$, the concept of the $k$-closure of a permutation group $G$ on a set $\Om$. The \emph{$k$-closure} $G^{(k),\Omega}$ of $G$ is the set of all  $g\in\Sym(\Om)$ (permutations of $\Om$) such that $g$ leaves invariant each $G$-orbit in the induced $G$-action on ordered $k$-tuples from $\Om$. The $k$-closure $G^{(k),\Omega}$ is a subgroup of $\Sym(\Om)$ containing $G$ \cite[Theorem 5.4]{Wielandt1969}, and a permutation group $G$ is said to be \emph{$k$-closed} if $G^{(k),\Omega}=G$.
Different faithful permutation representations of the same group $G$ may have quite different $k$-closures. For example, the symmetric group $S_3$ acts faithfully and intransitively on
$\{1,2,3,4,5\}$ with orbits $\{1,2,3\}$ and $\{4,5\}$, and in this action its $2$-closure is $S_3\times C_2$; while $S_3$ is $2$-closed in its natural action on $\{1,2,3\}$.

In 2016,
D.~F.~Holt\footnote{\texttt{mathoverflow.net/questions/235114/2-closure-of-a-permutation-group}}
suggested a stronger concept independent of the permutation representation, and this was studied first by Abdollahi and Arezoomand  in \cite{AA} in the case $k=2$.  For a positive integer $k$, a group $G$ is said to be \emph{totally $k$-closed} if  $G^{(k),\Omega}=G$ whenever $G$ is faithfully represented as a permutation group on $\Om$. The only totally $1$-closed group is the trivial group consisting of a single element (see Remark~\ref{rem-k1}), while Abdollahi and Arezoomand \cite[Theorem 2]{AA} showed that a finite nilpotent group is totally $2$-closed if and only if it is cyclic, or it is a direct product of a generalised quaternion group and a cyclic group of odd order.  Here we consider larger values of $k$.

For a permutation group $G\leq\Sym(\Omega)$, and $k\geq 2$, Wielandt~\cite[Theorem 5.8]{Wielandt1969} proved that
\begin{equation}\label{eq1}
G\leq G^{(k),\Omega}\leq G^{(k-1),\Omega}.
\end{equation}
Thus if $G$ is totally $(k-1)$-closed, then it is automatically totally $k$-closed. Moreover $G=G^{(k),\Omega}$ for sufficiently large $k$, since by \cite[Theorem 5.12]{Wielandt1969}, this holds whenever there exist $k-1$ points $\alpha_1,\dots,\alpha_{k-1}\in\Omega$ such that the only element of $G$ fixing each $\alpha_i$ is the identity. The inclusion \eqref{eq1} does suggest that the family of totally $k$-closed groups might be larger than that of totally $(k-1)$-closed groups. We show that this is the case, even for abelian groups.

\begin{theorem}\label{thm1}
Let $k$ be an integer with $k\geq2$. Then, for each prime $p$, there are infinitely many finite abelian $p$-groups which are totally $k$-closed but not totally $(k-1)$-closed.
\end{theorem}

The result of Abdollahi and Arezoomand shows that the finite totally $2$-closed abelian groups are precisely the cyclic groups. It turns out, also for larger values of $k$, that the total $k$-closure property for abelian groups is linked with the numbers of  cyclic direct factors in their direct decompositions.  A study of these decompositions leads to useful bounds, from which we deduce Theorem~\ref{thm1}.

According to the fundamental theorem for finite abelian groups, each nontrivial finite abelian group $G$ can be written as a  direct product $G=H_1\times\dots\times H_n$, for some $n\geq1$, such that each  $H_i\cong \mZ_{d_i}$, $d_1>1$, and $d_i|d_{i+1}$ for $1\leq i<n$. The integer $n$ and the $d_i$ are uniquely determined by $G$, up to the order of the factors. The $H_i$ are called the \emph{invariant factors} of $G$, and we write $n(G):=n$ for the number of invariant factors. We also have the primary decomposition of $G$ as
$G=\prod_{p\in\pi(G)} G_p$, where $\pi(G)$ is the set of primes dividing $|G|$ and $G_p$ is the (unique) Sylow $p$-subgroup of $G$. It is straightforward to see that $n(G)=\max_{p\in\pi(G)} n(G_p)$.   Our main result is the following theorem, from which we deduce Theorem~\ref{thm1}.

\begin{theorem}\label{thm2}
Let $G$ be a finite abelian group with $|G|>1$. Then $G$ is totally $(n(G)+1)$-closed, but is not totally $n(G)$-closed.
\end{theorem}

The following auxiliary assertion on the $k$-closure of the direct product abelian permutation $p$-groups may be of independent interest. It is proved in Section~\ref{prelim}, and is used in Section~\ref{proofs} to reduce the proof of Theorem~\ref{thm2} to the case of $p$-groups. For its statement it is convenient to use $\Syl(G)$ to denote the set of all Sylow subgroups of a group $G$; if $G$ is abelian, $\Syl(G)$ will consist of one Sylow $p$-subgroup for each prime $p\in\pi(G)$. 

\begin{theorem}\label{ABGrpClosure}
Let $G$ be a finite abelian permutation group on a set $\Omega$, and $k$ an integer, $k\geq 2$. Then
$G^{(k),\Omega}=\prod_{P\in\Syl(G)} P^{(k),\Omega}$.
\end{theorem}

The results in our short paper serve to raise a number of open questions, and we record a few here. The first relates to Theorem~\ref{thm2}. It would be interesting to have a generalisation of the classification by Abdollahi and Arezoomand \cite[Theorem 2]{AA} of nilpotent totally $2$-closed nilpotent groups for larger values of $k$.

\begin{problem}\label{p1}
For $k>2$ determine all finite nilpotent groups $G$ that are totally $k$-closed.
\end{problem}

As we noted above, the symmetric group $S_3$ is not totally $2$-closed.
Indeed, it was shown by Abdollahi, Arezoomand and Tracey~\cite[Theorem B]{AAT} a finite soluble group is
 totally $2$-closed if and only if it is nilpotent, hence known by \cite[Theorem 2]{AA}.
However it is not difficult to see that it is totally $3$-closed, since in every faithful permutation representation of $G=S_3$ on a set $\Omega$ there must be a $G$-orbit of length $3$ or $6$, and the stabiliser in $G$ of two points $\alpha, \beta$ from such an orbit is trivial. Hence by \cite[Theorem 5.12]{Wielandt1969}, $G=G^{(3),\Omega}$. As a first step it would be interesting to know which other non-nilpotent soluble groups are  totally $3$-closed.

\begin{problem}\label{p2}
Determine the finite soluble groups that are totally $3$-closed.
\end{problem}

For some time it was believed that all finite totally $2$-closed groups would be soluble, and it was somewhat surprising to the authors of \cite{AMPT} to discover that exactly six or the sporadic simple groups are totally $2$-closed, namely $J_1, J_3, J_4, Ly, Th, M$.

\begin{problem}\label{p3}
Find all the the  totally $3$-closed sporadic simple groups. More generally, for each sporadic simple group $G$ determine the least value of $k$ such that $G$ is totally $k$-closed.
\end{problem}

The classification of the finite nonabelian simple totally $2$-closed groups is still not complete, and we refer the reader to \cite{AMPT} for details of the status of this problem and other open questions about total $2$-closure.

\subsection*{Acknowledgements}
The authors would like to thank Natalia Maslova for her role in organising the 2020 Ural Workshop on Group Theory and Combinatorics. They are also grateful to an anonymous referee for some hepful comments on the exposition.
The first author is supported by Mathematical Center in Akademgorodok under agreement No. 075-15-2019-1613 with the Ministry of Science and Higher Education of the Russian Federation. The second author is supported by the Australian Research Council Discovery Project DP190100450.

\section{Preliminaries}\label{prelim}

In this section we give some background theory, and in particular we prove Theorem~\ref{ABGrpClosure}.
First we state two results of Wielandt for convenience.

\begin{theorem}\label{thm:W}{\rm [Wielandt, \cite[Theorem 5.6]{Wielandt1969}]}\quad Let $G\leq \Sym(\Omega)$, let $k\geq1$, and let $x\in\Sym(\Omega)$. Then $x\in G^{(k),\Omega}$ if and only if, for all $(\alpha_1, \dots, \alpha_k)\in\Omega^k$, there exists $g\in G$ such that $\alpha_i^x = \alpha_i^g$ for $i=1,\dots,k$.
\end{theorem}

\begin{theorem}\label{thm:W2}{\rm [Wielandt, \cite[Theorem 5.12]{Wielandt1969}]}\quad  Let $G\leq \Sym(\Omega)$ and $k\geq1$, and suppose that $\alpha_1,\ldots,\alpha_{k}\in\Omega$ such that $G_{\alpha_1\ldots \alpha_{k}}=1$. Then $G^{(k+1),\Omega}=G$.
\end{theorem}

Next we discuss total $1$-closure.

\begin{remark}\label{rem-k1}
{\rm
Suppose that $G$ is a finite totally $1$-closed group. Consider the regular representation of $G$ on $\Omega=G$. Since $G$ is transitive on $\Omega$ it follows from Theorem~\ref{thm:W} that $G^{(1),\Omega}=\Sym(\Omega)$. Thus, since $G$ is totally $1$-closed, it follows that $\Sym(\Omega)=G$ is regular, and hence $|G|\leq 2$. However, if $G=C_2$, then in the representation $G=\langle (12)(34)\rangle\leq\Sym(\Omega)$ on $\Omega=\{1,2,3,4\}$ we have  $G^{(1),\Omega}= \langle (12), (34)\rangle \ne G$. Hence $G=1$ is the only possibility.
}
\end{remark}

For a prime~$p\,|\, n$, the largest $p$-power divisor  of~$n$ is denoted by~$n_p$; if $\pi$ is a set prime divisors of $n$, then $n_\pi:=\prod_{p\in \pi}n_p$ denotes the $\pi$-part of $n$. Recall that, for a finite group $G$, $\pi(G)$ is the set of prime divisors of $|G|$. For $p\in\pi(G)$, we denote by $\Syl_p(G)$ the set of Sylow $p$-subgroups of $G$. For a subgroup $G\leq\Sym(\Omega)$ we denote by $\Orb(G)$ the set of $G$-orbits in~$\Omega$.

The proof of Theorem~\ref{ABGrpClosure} is developed using ideas from~\cite{ChurikovPonomarenko}.
First we present separately two lemmas as they are general results about finite nilpotent groups.

\begin{lemma}\label{SetwiseStabilizers}
Let $G$ be a finite nilpotent permutation group, let $p\in\pi(G)$, $k$ be a positive integer,  and $P\in\Syl_p(G)$. Let $\Delta_1,\ldots,\Delta_k\in \Orb(P)$, $\Delta=\bigcup_{i=1}^k \Delta_i$, and $L$ be the subgroup of $G$ consisting of all elements fixing each $\Delta_i$ setwise. Then $L^\Delta= P^\Delta.$
\end{lemma}

\begin{proof}
By the definition of $L$, the subgroup $P\leq L$, and hence $P^\Delta\leq L^\Delta$. We now prove the converse. Since $G$ is nilpotent, we have $G=P\times H$, where $H$ is the Hall $p'$-subgroup of~$G$. Let $g\in L$, so $g=xy$ for some (unique) $x\in P$ and $y\in H$. Since $P\leq L$, we have $y=x^{-1}g\in L$.

We claim that $y^\Delta=1$, or equivalently, that  $y^{\Delta_i}=1_{\Delta_i}$ for each $i=1,\ldots, k$.
Since $y\in H\leq C_G(P)$ it follows that, for each $i$,  $y^{\Delta_i}$ belongs to the centralizer $Z_i$ of the transitive group $P^{\Delta_i}\leq\Sym(\Delta_i)$, which is semiregular by~\cite[Theorem 3.2]{PS}. In particular  $|Z_i|$ divides~$|\Delta_i|$ which is a $p$-power, so $Z_i$ is a~$p$-group. Consequently, $y^{\Delta_i}$  is a $p$-element. Since $y\in H$ and $|H|$ is coprime to $p$, this implies that $y^{\Delta_i}=1$, for each $i$, and hence that $y^\Delta=1$,  proving the claim. Thus, $g^\Delta=(xy)^\Delta=x^\Delta y^\Delta=x^\Delta\in P^\Delta$, as required.
\end{proof}

\begin{lemma}\label{HallSubgrpOrbits}\emph{\cite[Lemma~2.4]{ChurikovPonomarenko}}
Let $G\le\Sym(\Omega)$, where $n=|\Omega|$ and $\pi\subseteq\pi(G)$. Suppose that $G$ is transitive and nilpotent, and let $H$ be a Hall $\pi$-subgroup of $G$. Then
\begin{itemize}
    \item[{\rm (1)}] the size of every $H$-orbit is equal to $n_\pi$, and
    \item[{\rm (2)}] $G$ acts on $\Orb(H)$; moreover, the kernel of this action is equal to~$H$.
\end{itemize}
\end{lemma}

\subsubsection*{Proof of Theorem~\ref{ABGrpClosure}}

Let $G$ be a finite abelian permutation group on a set $\Omega$, and let $k\geq 2$. Then by \cite[Theorem~5.8]{Wielandt1969} and~\cite[Exercise~5.26]{Wielandt1969}, $G^{(k),\Omega}$ is abelian, and $\pi(G^{(k),\Omega})=\pi(G)$. Let $p\in\pi(G)$, and let $P$ and $Q$ be the (unique) Sylow $p$-subgroups of $G$ and $G^{(k),\Omega}$ respectively. 

\medskip\noindent
\emph{Claim 1.} $P\leq P^{(k),\Omega}\leq Q$, and $\Orb(P)=\Orb(Q)$.

\smallskip\noindent
\emph{Proof of Claim 1.}\quad 
By~\cite[Theorem~5.8]{Wielandt1969} and~\cite[Exercise~5.28]{Wielandt1969}, the group $P^{(k),\Omega}$ is a $p$-group, and hence $P\leq P^{(k),\Omega}\leq Q$. It remains to prove that each $P$-orbit is a $Q$-orbit. Let $\Delta$ be  a $P$-orbit, and let $\Gamma$ be the $G$-orbit containing $\Delta$. By \eqref{eq1}, $G\leq G^{(k),\Omega} \leq G^{(1),\Omega}$, and hence $G^{(k),\Omega}$ has the same orbits as $G$ in $\Omega$. Thus $\Gamma$ is also a $G^{(k),\Omega}$-orbit, and hence the  $Q$-orbit $\Delta'$ containing $\Delta$ satisfies
$
\Delta\subseteq \Delta'\subseteq\Gamma.
$
The induced permutation groups $G^\Gamma$ and $(G^{(k),\Omega})^{\Gamma}$ are both transitive and abelian, so applying Lemma~\ref{HallSubgrpOrbits} to each of these groups with Hall subgroups $P^\Gamma, Q^\Gamma$, respectively, yields $|\Delta|=|\Gamma|_p=|\Delta'|$.  Thus $\Delta=\Delta'$, and Claim 1 is proved.
\qed

\medskip\noindent
\emph{Claim 2.} $P^{(k),\Omega}=Q$.

\smallskip\noindent
\emph{Proof of Claim 2.}\quad 
Let $(\alpha_1,\ldots,\alpha_k)\in\Omega^k$, and $g\in Q$. By Theorem~\ref{thm:W}, there exists $h\in G$ such that
$$
(\alpha_1,\ldots,\alpha_k)^g=(\alpha_1,\ldots,\alpha_k)^h.
$$
For each $i=1\ldots k$, let $\Delta_i$ be the $Q$-orbit containing $\alpha_i$. Then by Claim 1, each $\Delta_i$ is also a $P$-orbit. Since $P\unlhd G$, the group $G$ permutes the $P$-orbits, and for each $i$, since $\alpha_i^h=\alpha_i^g\in\Delta_i$, it follows that $h$ fixes each $\Delta_i$ setwise. Thus $h$ lies in the subgroup $L$ of Lemma~\ref{SetwiseStabilizers}, and setting $\Delta=\bigcup_{i=1}^k \Delta_i$, it follows from Lemma~\ref{SetwiseStabilizers} that $h^\Delta = u^\Delta$ for some $u\in P$. Thus
$$
(\alpha_1,\ldots,\alpha_k)^g=(\alpha_1,\ldots,\alpha_k)^h=(\alpha_1,\ldots,\alpha_k)^u.
$$
Since such an element $u\in P$ exists for each $k$-tuple of points and each $g\in Q$, it follows from Theorem~\ref{thm:W} that $g\in P^{(k),\Omega}$. Thus $Q\leq P^{(k),\Omega}$, and the reverse inclusion holds by Claim 1.
\qed

Now we complete the proof of Theorem~\ref{ABGrpClosure}. Since $G^{(k),\Omega}$ is abelian, $G^{(k),\Omega}$ is the direct product of its Sylow subgroups. Further,  for each $p\in\pi(G)$ it follows from Claim 2 that the unique Sylow $p$-subgroup of $G^{(k),\Omega}$ is   
$P^{(k),\Omega}$,  where $P$ is the unique Sylow $p$-subgroup of $G$.\qed

\section{Proof of the main results}\label{proofs}

Recall the definition of $n(G)$ given in Section~\ref{intro} for a finite abelian group $G$. We also set $N(G):=\sum_{p\in\pi(G)} n(G_p)$. If $G\leq \Sym(\Omega)$ then the \emph{base size} $b(G, \Omega)$ of $G$ is the smallest integer $b$ for which there exist $\alpha_1,\dots,\alpha_b\in\Omega$ such that $G_{\alpha_1\ldots \alpha_{b}}=1$. Such a set $\alpha_1,\dots,\alpha_b$ is called a \emph{base} of $G$.  Note that, by Theorem~\ref{thm:W2}, $G=G^{(b+1),\Omega}$,
where $b = b(G,\Omega)$.

\begin{lemma}\label{BaseNumber} Let $G$ be a finite abelian group and suppose that $G$ has a faithful permutation representation on a finite set $\Omega$. Then $b(G, \Omega)\leq N(G)$, and equality holds for some $\Omega$.
\end{lemma}

\begin{proof}
Let $G=\prod_{p\in\pi(G)} G_p$ with $\pi(G)$ the set of primes dividing $|G|$, and $G_p$ the Sylow $p$-subgroup of $G$, for $p\in\pi(G)$. Then, by the definition of $N(G)$ and the $n(G_p)$, $G$ has a direct decomposition $G=H_1\times\dots\times H_n$, with each $H_i$ nontrivial and cyclic of prime power order, and $n=N(G)$. For each $i$, $H_i$ acts regularly on $\Omega_i:=H_i$ by (right) multiplication, and $G$ acts faithfully on $\Omega:=\cup_{i=1}^n \Omega_i$ (where $H_j$ acts trivially on $\Omega_i$ for $i\ne j$). Thus the $G$-orbits in $\Omega$ are the sets $\Omega_i$, and for each $i$ the subgroup $H_i$ acts nontrivially only on the orbit $\Omega_i$. Thus each base must contain a point from each of the $G$-orbits. It follows that the base size equals $N(G)$ for this faithful permutation representation of $G$.

Now consider an arbitrary faithful permutation representation of $G$, that is, suppose that $G\leq\Sym(\Omega)$. We prove by induction on $N(G)$ that $G$ has base size at most $N(G)$.  Now
 $H_1=\langle h_1\rangle\cong \mZ_{p^a}$, for some prime $p$ and positive integer $a$, and
as $G$ acts faithfully on $\Omega$ there exists $\alpha\in\Omega$ which is not fixed by $h_1^{p^{a-1}}$.
This implies that $G_\alpha\cap H_1=1$.
If $N(G)=1$ then $G=H_1$ is a cyclic $p$-group, and $G_\alpha=1$, so $\{\alpha\}$ is a base.
Assume now that $N(G)\geq 2$ and that the assertion holds for groups $X$ with $N(X)<N(G)$.
Since $G_\alpha\cap H_1=1$, we have $G_\alpha\cong (G_\alpha H_1)/H_1 \leq G/H_1\cong \prod_{i=2}^{n}H_i$ so $N(G_\alpha)\leq n-1=N(G)-1$, and hence by induction, $G_\alpha$ has a base $\alpha_1,\dots,\alpha_s$ in $\Omega\setminus\{\alpha\}$ with $s\leq N(G)-1$. Then $\alpha_1,\dots,\alpha_s,\alpha$ is a base for $G$ in $\Omega$, and the result follows by induction.
\end{proof}

We now prove Theorem~\ref{thm2} in the case of $p$-groups. The second part of the lemma is proved using a construction developed from ideas in the book of Chen and Ponomarenko \cite[Proposition 2.2.26]{CP}. An element $\tau\in\Sym(\Omega)$ is called a \emph{cycle} if it is not the identity and has exactly one cycle of length greater than $1$ in its disjoint cycle representation; the length of this cycle is denoted $|\tau|$. Two cycles are said to be \emph{independent} if the sets of points they move are disjoint.

\begin{lemma}\label{PGroupsCase}
Let $G$ be a finite abelian $p$-group with $|G|>1$. Then $G$ is totally $(n(G)+1)$-closed, but is not totally $n(G)$-closed.
\end{lemma}

\begin{proof}
Since $G$ is an abelian $p$-group,  $N(G)=n(G)$. By Lemma~\ref{BaseNumber}, if $G$ is faithfully represented as a subgroup of $\Sym(\Omega)$, then $b:=b(G,\Omega)\leq n(G)$, and by Theorem~\ref{thm:W2}, $G=G^{(b+1),\Omega}$. It follows from \eqref{eq1} that $G=G^{(n(G)+1),\Omega}$. Since this holds for all faithful permutation representations of $G$, $G$ is totally $(n(G)+1)$-closed.
\medskip

As discussed in Section~\ref{intro}, $G\cong\mZ_{d_1}\times\mZ_{d_2}\times\ldots\times\mZ_{d_n}$, with
$d_1>1$,  $d_i|d_{i+1}$ for $1\leq i<n$, and $n=n(G)$. Let $\Omega$ be a set of size $d_1+\sum_{i=1}^nd_i$, and let $\tau_0,\tau_1,\ldots,\tau_n\in\Sym(\Omega)$ be pairwise independent cycles on $\Omega$ such that $|\tau_0|=d_1$, and $|\tau_i|=d_i$ for $i=1\ldots n$. Let $H_1=\langle \tau_0\tau_1 \rangle$ and $H_i = \langle \tau_0^{-1}\tau_i \rangle$ for $i=2\ldots n$, and let
$$
H=\langle H_1,\ldots,H_n \rangle.
$$
We claim that $H\cong G$. Indeed, the groups $H_i$ commute, and  an easy proof by induction on $n$ shows that
$$
H_i \cap \langle H_1,\ldots,H_{i-1},H_{i+1},\ldots,H_n \rangle = 1,\ \text{for}\ i=1\ldots n.
$$
Thus $H = H_1 \times \ldots \times H_n$, with $H_i \cong \mZ_{d_i}$ for $i=1\ldots n$, proving the claim.

Now we will use Theorem~\ref{thm:W} to show that $\tau_0\in H^{(n),\Omega}$. Let $(\alpha_1,\ldots,\alpha_n) \in \Omega^n$, and for  $i=0,\ldots, n$, let $\Delta_i$ denote the set of points of $\Omega$ moved by $\tau_i$, so that $\{\Delta_0,\ldots,\Delta_n\}$ is the set of $H$-orbits in $\Omega$. Since $H$ has $n+1$ nontrivial orbits, there exists $k\in\{0,1,\dots,n\}$ such that $\Delta_k\cap\{\alpha_1,\dots,\alpha_n\}=\varnothing$.  Define a permutation $\tau$ as follows:
\[
\tau=\begin{cases}
1 \text{,\quad if $k=0$},\\
\tau_0\tau_k^{-1} \text{, if $1\leq k\leq n$}.
\end{cases}
\]
By definition, $\tau\in H$. If $\tau=1$, then both $\tau$ and $\tau_0$ fix each of the $\alpha_i$ so
$(\alpha_1,\ldots,\alpha_n)^{\tau_0}=(\alpha_1,\ldots,\alpha_n)^\tau$. On the other hand, if $\tau=\tau_0\tau_k^{-1}$ for some $k$, then $\tau$ and $\tau_0$ induce the same permutation on $\Omega\setminus\Delta_k$,
and again we have $(\alpha_1,\ldots,\alpha_n)^{\tau_0}=(\alpha_1,\ldots,\alpha_n)^{\tau}$. Thus, by
Theorem~\ref{thm:W}, $\tau_0\in H^{(n),\Omega}$. By the construction, $\tau_0\notin H$, and hence $H\ne H^{(n),\Omega}$. Thus $G$ is not totally $n$-closed.
\end{proof}

\begin{remark}
{\rm
Theorem~\ref{thm1} follows from
Lemma~\ref{PGroupsCase} since, for each integer~$k\geq 2$, there are infinitely many finite abelian $p$-groups with $k$ invariant factors.
}
\end{remark}

Finally we prove Theorem~\ref{thm2} for an arbitrary finite abelian group $G$ with $|G|>1$. Suppose that $G$ is faithfully represented on a set $\Omega$. Since $n(G)=\max_{p\in\pi(G)} n(G_p)$, every Sylow subgroup $G_p$ of $G$ is $(n(G)+1)$-closed by Lemma~\ref{PGroupsCase}, and hence, by Theorem~\ref{ABGrpClosure}, we have $G^{(n(G)+1),\Omega}=G$. Thus $G$ is totally $(n(G)+1)$-closed.

Set $n :=n(G)$. If $n=1$ then, since $|G|>1$, it follows from Remark~\ref{rem-k1} that $G$ is not totally $1$-closed. Thus we may assume that $n\geq 2$. Now $n(G)=\max_{p\in\pi(G)} n(G_p)$, and hence we have $n=n(G_q)$ for some $q\in\pi(G)$.
By Lemma~\ref{PGroupsCase}, $G_q$ is not totally $n$-closed, so there exists a set $\Omega_q$ such that $G_q$ acts faithfully on $\Omega_q$ and $G^{(n),\Omega_q}\ne G_q$. There is nothing further to prove if $G=G_q$ so we may assume that $|\pi(G)|\geq2$.
For each $p\in\pi(G)\setminus\{q\}$, let
$\Omega_p =G_p$, and consider $G_p$ acting regularly on $\Omega_p$ by right multiplication. Thus $G$ acts faithfully on $\Omega:=\cup_{p\in\pi(G)} \Omega_p$. Since $n\geq 2$, it follows from Theorem~\ref{ABGrpClosure} that
$$
G^{(n),\Omega}=\prod_{p\in\pi(G)} (G_p)^{(n),\Omega_p}=(G_q)^{(n),\Omega_q}\times\prod_{\substack{p\in\pi(G) \\ p\neq q}} (G_p)^{(n),\Omega_p},
$$
which is not equal to $G$, because $(G_q)^{(n),\Omega_q}>G_q$ and for every $p\in\pi(G), p\neq q$ the group $G_p$ is $n$-closed as a regular group. Thus, $G$ is not totally $n$-closed, and the proof of Theorem~\ref{thm2} is complete.

\end{document}